\documentclass[10pt,a4paper,twoside]{amsart}
\usepackage[T1]{fontenc}
\usepackage[utf8]{inputenc}
\usepackage[english]{babel}
\usepackage{amsmath}
\usepackage{amsthm}
\usepackage{amssymb}
\usepackage{amsfonts}
\usepackage{amsxtra}
\usepackage{enumerate}
\usepackage{verbatim}
\usepackage{color}
\usepackage[mathscr]{eucal}
\usepackage{graphicx}

\newtheorem*{theo}{Theorem}
\newtheorem*{conj}{Conjecture}
\newtheorem{qu}{Question}

\newtheorem{theorem}{Theorem}[section]
\newtheorem{definition}[theorem]{Definition}
\newtheorem{lemma}[theorem]{Lemma}
\newtheorem{cor}[theorem]{Corollary}

\newtheorem{prop}[theorem]{Proposition}
\newtheorem*{remark}{Remark}
\theoremstyle{definition}
\newtheorem{exmp}[theorem]{Example}

\let\phi=\varphi

\newcommand{\integ}[4]{\int\limits_{#1}^{#2}#3\, d#4}

\newcommand{\abb}[3]{#1\colon #2\rightarrow #3}
\newcommand{\real}[1]{\mathbb{R}^{#1}}
\newcommand{\rem}[1]{}

\DeclareFontFamily{U}{mathb}{\hyphenchar\font45}
\DeclareFontShape{U}{mathb}{m}{n}{
<-6> mathb5 <6-7> mathb6 <7-8> mathb7
<8-9> mathb8 <9-10> mathb9
<10-12> mathb10 <12-> mathb12
}{}
\DeclareSymbolFont{mathb}{U}{mathb}{m}{n}
\DeclareMathSymbol{\llcurly}{\mathrel}{mathb}{"CE}
\DeclareMathSymbol{\ggcurly}{\mathrel}{mathb}{"CF}

\begin{document}

\title{Lorentzian distance functions in contact geometry}
\author{Jakob Hedicke}
\address{Ruhr-Universit\"at Bochum\\ Fakult\"at f\"ur Mathematik\\ Universit\"atsstra\ss e 150\\ 44801 Bochum, Germany}
\email{Jakob.Hedicke@ruhr-uni-bochum.de}
\date{\today}

\begin{abstract}
An important tool to analyse the causal structure of a Lorentzian manifold is given by the Lorentzian distance function.
We define a class of Lorentzian distance functions on the group of contactomorphisms of a closed contact manifold depending on the choice of a contact form.
These distance functions are continuous with respect to the Hofer norm for contactomorphisms defined by Shelukhin (\cite{Egor}) and finite if and only if the group of contactomorphisms is orderable.
To prove this we show that intervals defined by the positivity relation are open with respect to the topology induced by the Hofer norm.
For orderable Legendrian isotopy classes we show that the Chekanov-type metric defined in \cite{Rosen} is non-degenerate.
In this case similar results hold for a Lorentzian distance functions on Legendrian isotopy classes.
This leads to a natural class of metrics associated to a globally hyperbolic Lorentzian manifold such that its Cauchy hypersurface has a unit co-tangent bundle with orderable isotopy class of the fibres.
\end{abstract}
\maketitle

\section{Introduction}
Consider a closed co-oriented contact manifold $(M,\xi)$, i.e. an $(2n+1)$-dimen\-sional smooth manifold $M$ with a hyperplane distribution $\xi\subset TM$ that is the kernel of a $1$-form $\alpha$ such that $\alpha\wedge d\alpha^n$ is nowhere vanishing.
Denote by $\mathrm{Cont}_0(M)$ the identity component of the group of contactomorphisms, i.e. the group of diffeomorphisms preserving $\xi$ that are isotopic to $id_M$ through contactomorphisms.
Fore details see e.g. \cite{Geiges}.

In \cite{Eliashberg} Eliashberg and Polterovich introduced the concept of positivity on $\mathrm{Cont}_0(M)$ (see also \cite{Bhupal}).
An isotopy $\phi_t$ of contactomorphisms is called \textit{positive} if the contact vector field 
$$X_t^{\phi}\circ \phi_t:= \frac{d}{ds}|_{s=t}\phi_s$$
satisfies $\alpha(X_t^{\phi})>0$ for all $t$.
Note that this definition only depends on the co-orientation defined by the contact form $\alpha$.
Similarly $\phi_t$ is called \textit{non-negative} if $\alpha(X_t^{\phi})\geq 0$.
This induces two relations on $\mathrm{Cont}_0(M)$ by
$$\phi \llcurly \psi :\Leftrightarrow \text{ there exists a positive isotopy from $\phi$ to $\psi$}$$
and
$$\phi \preccurlyeq \psi :\Leftrightarrow \text{ there exists a non-negative isotopy from $\phi$ to $\psi$}.$$
The relations $\llcurly$ and $\preccurlyeq$ turn $\mathrm{Cont}_0(M)$ into a causal space (see \cite{Kunzinger}).

The properties of the relations on $\mathrm{Cont}_0(M)$ resemble properties of the chronological and causal relation in Lorentzian geometry.
Moreover as we will show, one can define an analogue of the Lorentzian distance function on the group $\mathrm{Cont}_0(M)$.

Let $(N,g)$ be a smooth time-oriented Lorentzian manifold, that is a manifold $N$ with a smooth pseudo-Riemannian metric of signature $(1,n)$ and the choice of a vector field $X$ such that $g(X,X)<0$.
A tangent vector $v\in TN$ is called \textit{future pointing timelike} if $g(v,v)<0$ and $g(v,X)< 0$ and future pointing causal if $g(v,v)\leq 0$ and $g(v,X)< 0$.
A smooth curve $\gamma(t)$ is called future pointing timelike (causal) if $\gamma'(t)$ is future pointing timelike (causal) for all $t$.
This induces two relations on $N$.
The chronological relation 
$$p \ll q:\Leftrightarrow \text{ there exists a future pointing timelike curve from $p$ to $q$}$$
and the causal relation
$$p \leq q:\Leftrightarrow \text{ there exists a future pointing causal curve from $p$ to $q$}.$$
Most causal properties of $(N,g)$ can be phrased in terms of the chronological (causal) future/past of the points in $N$:
\begin{align*}
I^+(p):=\{q\in N|p\ll q\}&, I^-(p):=\{q\in N|q\ll p\}\\
J^+(p):=\{q\in N|p\leq q\}&, J^-(p):=\{q\in N|q\leq p\}.
\end{align*}

A natural topology on $N$ related to $g$ is the interval (Alexandrov) topology whose basis is given by the open 'intervals' $I^+(p)\cap I^-(q)$ of the chronological relation.
The interval topology coincides with the manifold topology of $N$ iff the Lorentzian manifold is \textit{strongly causal}, i.e. iff for every open $U\subset N$ there exists a causally convex open $V\subset U$ (\cite{Minguzzi2}).

Analogously to the length of a curve $\gamma$ in a Riemannian manifold, one can define a \textit{Lorentzian length (eigentime)} for causal curves by 
$$L_g(\gamma):= \integ{a}{b}{\sqrt{-g(\gamma'(t),\gamma'(t))}}{t}.$$
The \textit{Lorentzian distance function} between two points is then defined by
\begin{align*}
&\abb{\tau_g}{N\times N}{[0,\infty]}\\
(p,q)&\mapsto \left\lbrace \begin{array}{cc}
\sup\{L_g(\gamma)\}, &\text{if } p\leq q \\ 
0, & \text{otherwise}
\end{array} \right..
\end{align*}
Here the supremum is taken over all future pointing causal curves connecting $p$ and $q$.

The Lorentzian distance function satisfies (see \cite{Beem})
\begin{itemize}
\item[(i)] $\tau_g(p,q)>0$ if and only if $p\ll q$.
\item[(ii)] $\tau_g$ is lower semi-continuous.
\item[(iii)] $\tau_g(p,q)\geq \tau_g(p,r)+\tau_g(r,q)$ for $p\leq r \leq q$.
\end{itemize}
If $(N,g)$ is strongly causal the metric $g$ can be recovered from $\tau_g$ and the Lorentzian distance reflects many causal properties of $(N,g)$.

A strongly causal time oriented Lorentzian manifold $(N,g)$ is called \textit{globally hyperbolic} if $J^+(p)\cap J^-(q)$ is compact for all $p,q\in N$.
In this case $N$ contains a smooth Cauchy hypersurface $\Sigma$, i.e. a surface that is intersected in a unique point by every inextendible causal curve.
In particular $N$ is diffeomorphic to $\real{}\times \Sigma$ (\cite{Minguzzi2}).

\begin{theo}[{{\cite[Corollary 4.7.]{Beem}}}]
If $(N,g)$ is globally hyperbolic, then $\tau_g$ is finite and continuous.
\end{theo}

\begin{theo}[{{\cite[Theorem 4.30.]{Beem}}}]
A strongly causal manifold $(N,g)$ is globally hyperbolic if and only if $\tau_{\tilde{g}}$ is finite for all $\tilde{g}$ in the conformal class of $g$, i.e. for all $\tilde{g}$ with $\tilde{g}=e^fg$  for some smooth function $f$. 
\end{theo}

As pointed out for instance in \cite{Kunzinger} Lorentzian distance functions can be used to explore the causal properties of more general causal spaces.
There the authors consider causal spaces and functions satisfying the properties (i)-(iii) (here lower semi-continuity is with respect to some background metric).

For closed $(M,\xi)$ we will define a class of Lorentzian distance functions on $\mathrm{Cont}_0(M)$  for the relations $\llcurly$ and $\preccurlyeq$ inspired by Shelukhins definition of the Hofer norm (\cite{Egor}).
The Lorentzian distance functions depend on the choice of a contact form.
They are not conjugation invariant but continuous with respect to the Hofer-Shelukhin norm and share some of its properties.
To prove this we show that the interval topology on $\mathrm{Cont}_0(M)$ introduced in \cite{Chernov19} is coarser than the topology induced by the Hofer-Shelukhin norm.
The same method that proves continuity of the Lorentzian distance function can be used to answer \cite[Question 18]{Egor} if $\mathrm{Cont}_0(M)$ is orderable in the sense of \cite{Chernov16}.

In the case of certain unit co-tangent bundles we use spectral invariants for Legendrians to give an estimate of the Lorentzian distance function in terms of the Hofer-Shelukhin norm.

Analogously to the relations on $\mathrm{Cont}_0(M)$, positivity defines two relations on the Legendrian isotopy classes of $(M,\xi)$ (see e.g. \cite{Chernov16}, \cite{Colin} for details).
In Section 5 we will define Lorentzian distance functions for isotopy classes of closed Legendrians.
For orderable Legendrian isotopy classes we give a simple proof of \cite[Conjecture 1.10]{Rosen} proving that the Chekanov-type metric defined by Rosen and Zhang is non-degenerate.
In this case the same results like for contactomorphisms hold, i.e. the Lorentzian distance functions are finite and continuous with respect to the Rosen-Zhang metric.

The results about the isotopy class of the fibre of unit co-tangent bundles are of particular interest to Lorentzian geometry.
Unlike the Riemannian case on Lorentzian manifolds there is no canonical way to define a metric.
As observed in \cite{Low} a unit co-tangent bundle $ST^{\ast}\Sigma$ is contactomorphic to the space of null geodesics of globally hyperbolic manifolds with Cauchy hypersurface $\Sigma$.
If the isotopy class of the fibre of the unit co-tangent bundle of the Cauchy surface is orderable, the points in the Lorentzian manifold can be identified with Legendrians isotopic to the fibres such that the identification respects the relations on the manifold and the Legendrian isotopy class (see \cite{Chernov16,Chernov19}).
In this case the Lorentzian distance and the Rosen-Zhang metric on the Legendrian isotopy class naturally induce a Lorentzian distance and a metric on the globally hyperbolic manifold.

\textbf{Acknowledgements.}

I am grateful to Alberto Abbondandolo,  Stefan Nemirovski, Daniel Rosen, Stefan Suhr and Kai Zehmisch for many useful discussions and their support.
This research is supported by the SFB/TRR 191 ``Symplectic Structures in Geometry, Algebra and Dynamics'', funded by the Deutsche Forschungsgemeinschaft (Projektnummer 281071066 – TRR 191).

\section{Main Results}
 
Let $(M,\xi)$ be a closed co-oriented contact manifold.
Given a contact form $\alpha$ Shelukhin \cite{Egor} defined a non-conjugation invariant norm on  $\mathrm{Cont}_0(M)$ by 
$$|\phi|_{\alpha}:=\inf\left\lbrace\left.\integ{0}{1}{\max\limits_M|\alpha(X_t^{\phi})|}{t}\right\vert \phi_t \text{ isotopy with } \phi_0=id_M, \phi_1=\phi\right\rbrace.$$
The norm $|\cdot|_{\alpha}$ has the following properties

\begin{theo}[\cite{Egor}]
The norm $|\cdot|_{\alpha}$ satisfies
\begin{itemize}
\item[(i)]$|\phi|_{\alpha}=0\Leftrightarrow \phi=id_M$.
\item[(ii)]$|\phi\psi|_{\alpha}\leq |\phi|_{\alpha}+|\psi|_{\alpha}$.
\item[(iii)]$|\phi^{-1}|_{\alpha}=|\phi|_{\alpha}$.
\item[(iv)]$|\psi\phi\psi^{-1}|_{\alpha}=|\phi|_{\psi^{\ast}\alpha}$.
\end{itemize}
\end{theo}

\begin{remark}
In the following we will work with the metric $d_{\alpha}(\phi,\psi):=|\psi^{-1}\phi|_{\alpha}$.
Note that the metrics induced by two different contact forms are equivalent (\cite{Egor}).
Therefore they induce the same topology on $\mathrm{Cont}_0(M)$.
\end{remark}

For a contactomorphism $\phi\in \mathrm{Cont}_0(M)$ define the sets $I^+(\phi), I^-(\phi), J^+(\phi)$ and $J^-(\phi)$ with respect to the relations $\llcurly$ and $\preccurlyeq$ analogously to the Lorentzian case.
A natural topology on $\mathrm{Cont}_0(M)$ introduced in \cite{Chernov19} is then given by the interval topology, i.e. the topology induced by sets of the form $I^+(\phi)\cap I^-(\psi)$.

\begin{prop}\label{propinterval}
The interval topology is coarser than the topology induced by $d_{\alpha}$, in particular the sets $I^{\pm}(\phi)$ are open with respect to $d_{\alpha}$.
\end{prop}

\begin{definition}\label{def1}
Given a contact form $\alpha$ on a closed contact manifold $(M,\xi)$ define $\abb{\tau_{\alpha}}{\mathrm{Cont}_0(M)\times \mathrm{Cont}_0(M)}{[0,\infty]}$ by 
\begin{align*}
&\tau_{\alpha}(\phi,\psi):=\left\lbrace \begin{array}{cc}
\sup\left\lbrace\integ{0}{1}{\min\limits_M\alpha(X_t^{\phi})}{t}\right\rbrace, &\text{if } \phi\preccurlyeq \psi \\ 
0, & \text{otherwise}
\end{array} \right..
\end{align*}
Here the supremum is taken over all non-negative paths $\phi_t$ with $\phi_0=\phi$ and $\phi_1=\psi$.
\end{definition}

Like the metric $d_{\alpha}$ the function $\tau_{\alpha}$ fails to be left invariant.
In \cite{Burago} and \cite{Fraser} it is shown that any conjugation invariant norm on $\mathrm{Cont}_0(M)$ is discrete in the sense that any contactomorphism that is not the identity has norm greater than some positive constant.
It is a natural question to ask weather a similar result holds for Lorentzian distance functions.

\begin{conj}\label{con1}
Let $\abb{\tau}{\mathrm{Cont}_0(M)\times\mathrm{Cont}_0(M)}{[0,\infty]}$ be a map satisfying $\tau(\phi,\psi)>0$ iff $\phi\llcurly \psi$ and $\tau(\phi_1,\phi_{2})\geq \tau(\phi_1,\psi)+\tau(\psi,\phi_2)$ for $\phi_1\preccurlyeq\psi\preccurlyeq\phi_2$.
Assume that $\tau$ is lower semi-continuous with respect to the interval topology and bi-invariant.
Then 
\begin{align*}
&\tau(\phi,\psi)=\left\lbrace \begin{array}{cc}
\infty, &\text{if } \phi\llcurly \psi \\ 
0, & \text{otherwise}
\end{array} \right..
\end{align*}
\end{conj}

\begin{remark}
The reverse triangle inequality immediately implies that there are no finite discrete Lorentzian distance functions since such a function has to be strictly increasing along positive paths of contactomorphisms.
\end{remark}

In section \ref{sleg} we will prove a version of this conjecture for Legendrian isotopy classes.

\begin{theorem}\label{thm1}
The map $\tau_\alpha$ satisfies
\begin{itemize}
\item[(i)]$\tau_{\alpha}(\phi,\psi)>0\Leftrightarrow \phi \llcurly \psi$.
\item[(ii)]$\tau_{\alpha}$ is continuous with respect to the interval topology and the topology induced by $d_{\alpha}$.
\item[(iii)]$\tau_{\alpha}(\phi_1,\phi_{2})\geq \tau_{\alpha}(\phi_1,\psi)+\tau_{\alpha}(\psi,\phi_2)$ for $\phi_1\preccurlyeq\psi\preccurlyeq\phi_2$.
\item[(iv)]$\tau_{\alpha}(\psi\phi_1,\psi\phi_2)=\tau_{\psi^{\ast}\alpha}(\phi_1,\phi_2)$.
\end{itemize}
\end{theorem}

\begin{qu}
Theorem \ref{thm1} shows that ${(\mathrm{Cont}_0(M),d_{\alpha},\tau_{\alpha})}$ is a Lorentzian pre-length space in the sense of \cite{Kunzinger}.
Are there contact manifolds such that \newline ${(\mathrm{Cont}_0(M),d_{\alpha},\tau_{\alpha})}$ has the structure of a Lorentzian length space?
\end{qu}

To answer the question one would need to show the local existence of non-negative paths maximizing the integral in Definition \ref{def1}, i.e. the existence of paths connecting two $d_{\alpha}$-close contactomorphisms $\phi \preccurlyeq\psi$ so that their 'Lorentzian length' coincides with $\tau_{\alpha}(\phi,\psi)$.

Call $\mathrm{Cont}_0(M)$ \textit{orderable} if $\preccurlyeq$ defines a partial order on $\mathrm{Cont}_0(M)$, i.e. if there are no non-negative loops.
By \cite[Proposition 2.1.B]{Eliashberg} this is equivalent to the non existence of a positive loop connecting $id_M$ to itself.
Note that the proof in \cite{Eliashberg} is given for contractible non-negative loops but works analogously for non contractible loops.

The following Corollary of Lemma \ref{lemimportant} answers \cite[Question 18]{Egor} if $\mathrm{Cont}_0(M)$ is orderable.

\begin{cor}\label{thmq18}
Suppose that $(M,\xi)$ is closed and $\mathrm{Cont}_0(M)$ orderable.
Denote by $\phi_t^{\alpha}$ the Reeb flow of a contact form $\alpha$.
Then $|\phi_t^{\alpha}|_{\alpha}=|t|$, in particular
$$\sup\limits_{\phi\in\mathrm{Cont}_0(M)}|\phi|_{\alpha}=\infty.$$
\end{cor}

Similarly one can prove

\begin{cor}\label{thm2}
Let $(M,\xi)$ be a closed contact manifold.
The group $\mathrm{Cont}_0(M)$ is orderable if and only if $\tau_{\alpha}(\phi,\psi)<\infty$ for all $\phi, \psi\in \mathrm{Cont}_0(M)$.
In this case, if $\phi_t^{\alpha}$ denotes the Reeb-flow with respect to $\alpha$, then
$$\tau_{\alpha}(\phi,\phi_t^{\alpha}\phi)=t$$
for all $ t\geq 0$ and all $\phi\in \mathrm{Cont}_0(M)$.
\end{cor}

\begin{remark}
The fact that orderability is equivalent to the non existence of positive loops and to the finiteness of $\tau_{\alpha}$ is a huge difference to Lorentzian geometry.
In \cite{Minguzzi2} there are given examples of manifolds that contain lightlike loops but no timelike loops.
Moreover there are many examples of strongly causal spacetimes such that the Lorentzian distance is not finite (see e.g. \cite[Theorem 4.30.]{Beem}).
\end{remark}

For $0<\epsilon<\infty$ consider the sets
\begin{align*}
B_{\alpha}^+(\phi,\epsilon):=\{\psi\in \mathrm{Cont}_0(M)|0<\tau_{\alpha}(\phi,\psi)<\epsilon\},\\
 B_{\alpha}^-(\phi,\epsilon):=\{\psi\in \mathrm{Cont}_0(M)|0<\tau_{\alpha}(\psi,\phi)<\epsilon\}.
\end{align*}

\begin{cor}
The sets $B_{\alpha}^+(\phi,\epsilon_1)\cap B_{\alpha}^-(\psi,\epsilon_2)$ form a basis of the interval topology.
\end{cor}

\begin{proof}
Theorem \ref{thm1} implies that $\tau_{\alpha}$ is continuous with respect to the interval topology.
Then the proof works analogous to the case of a strongly causal Lorentzian manifold $(M,g)$ such that $\tau_g$ is continuous (see \cite[Proposition 4.31.]{Beem}).
\end{proof}

Now consider the case when $M$ is a spherical co-tangent bundle, i.e. $M$ is the quotient bundle $T^{\ast}N\setminus\{0\}/\real{}_{>0}$ for some smooth manifold $N$.
Here $\real{}_{>0}$ acts on $T^{\ast}N\setminus\{0\}$ by positive fibrewise homotheties.
The spherical co-tangent bundle naturally carries a contact structure $\xi_{st}$:
Consider the canonical Liouville form $\lambda$ on $T^{\ast}N$.
The $\real{}_{>0}$-action induces a Liouville vector field $Y$ on $T^{\ast}N\setminus\{0\}$, i.e. a vector field that satisfies $d\lambda(Y,\cdot)=\lambda$.
Since $Y$ is tangent to the $\real{}_{>0}$-action, the kernel of $\lambda$ projects to a contact structure $\xi_{st}$ on the quotient (see \cite{Geiges}).

Using estimates for spectral invariants of Legendrians on jet-spaces (\cite{Zapolsky}) we will in section \ref{sec4} show the following estimate for $\tau_{\alpha}$:

\begin{theorem}\label{thm3}
Assume that $(M,\xi)\cong (ST^{\ast}N,\xi_{st})$, where $N$ is closed and smoothly covered by an open subset of $\real{n}$.
Then there exists a constant $C_{\alpha}$ depending on $\alpha$ such that for all $\phi, \psi \in \mathrm{Cont}_0(M)$
$$\tau_{\alpha}(\phi,\psi)\leq C_{\alpha} d_{\alpha}(\phi,\psi).$$
\end{theorem}

In general the interval topology is strictly coarser than the topology induced by $d_{\alpha}$.
If $\mathrm{Cont}_0(M)$ is not orderable it is not even Hausdorff.
With the assumptions of Theorem \ref{thm3} (except for $N$ being closed) \cite{Chernov19} used similar methods to show that the interval topology on $\mathrm{Cont}_0(M)$ is Hausdorff.

\begin{qu}
Assume that $\mathrm{Cont}_0(M)$ is orderable.
Is $\mathrm{Cont}_0(M)$ then strongly causal in the sense that the interval topology coincides with the topology induced by $d_{\alpha}$?
\end{qu}

\begin{qu}
As pointed out above in Lorentzian geometry the strongly causal manifolds with a Lorentzian distance finite and continuous on the conformal class are globally hyperbolic.
Do Theorem \ref{thm1} and Corollary \ref{thm2} imply that if $\mathrm{Cont}_0(M)$ is orderable other properties of globally hyperbolic manifolds like the existence of time functions or Cauchy surfaces transfer to $\mathrm{Cont}_0(M)$?
\end{qu}

\section{Proofs}

The following result is the key Lemma to prove Theorem \ref{thm1} and many other results.

\begin{lemma}\label{lemimportant}
Let $\phi_t$ be a path of contactomorphisms with 
$$\integ{0}{1}{\min\limits_M\alpha(X_t^{\phi})}{t}=\epsilon.$$
Then for any $\delta>0$ there exists a path $\psi_t$ with $\psi_0=\phi_0$ and $\psi_1=\phi_1$ such that for all $t\in[0,1]$
$$\min\limits_M\alpha(X_t^{\psi})\in (\epsilon-\delta,\epsilon+\delta).$$
\end{lemma}

\begin{proof}[Proof of Lemma \ref{lemimportant}]
Let $\abb{\tau}{[0,1]}{\real{}}$ be defined by $\tau(0)=0$ and 
$$\tau'(t)=-\min\limits_M\alpha(X_t^{\phi})+\epsilon.$$
Note that $\tau'$ is continuous since $\alpha(X_t^{\phi})$ is smooth, i.e. $\tau(t)$ is $\mathcal{C}^1$.
Denote by $\phi_t^{\alpha}$ the Reeb-flow of the contact form $\alpha$.
Define $\tilde{\psi}_t:=\phi^{\alpha}_{\tau(t)}\circ \phi_t$.
Then $\tilde{\psi}_t$ is a $\mathcal{C}^1$ path of contactomorphisms with $\tilde{\psi}_0=\phi_0$ and $\tilde{\psi}_1=\phi_1$ since $\tau(0)=0$ and 
$$\tau(1)=-\integ{0}{1}{\min\limits_M\alpha(X_t^{\phi})}{t}+\epsilon =0.$$
Moreover 
$$\min\limits_M\alpha(X_t^{\tilde{\psi}})=\min\limits_M\alpha(X_t^{\phi})+\tau'(t)=\epsilon.$$
Here we used that $\phi_{\tau(t)}^{\alpha}$ is a strict contactomorphism for all $t$, i.e. $\phi_{\tau(t)}^{\alpha}{}^{\ast}\alpha=\alpha$.
Approximating $\tau$ with smooth functions with fixed endpoints we get $\psi_t$ with the desired properties.
\end{proof}

\begin{proof}[Proof of Proposition \ref{propinterval}]
We show that $I^{+}(\phi)$ is open with respect to $d_{\alpha}$.
The proof works analogously for $I^-(\phi)$.
Let $\psi\in I^+(\phi)$ and $\phi_t$ be a positive path with $\phi_0=\phi$ and $\phi_1=\psi$.
Choose $\epsilon>0$ such that $\min\limits_M\alpha(X_t^{\phi})>\epsilon$.
There exists a smooth family of positive functions $\rho_t$ such that $(\phi_t\circ\psi^{-1})^{\ast}\alpha =\rho_t\alpha$.
Define
$$\alpha_0:=\max\limits_{[0,1]\times M}\rho_t \alpha.$$
Take $\tilde{\psi}$ with $d_{\alpha_0}(\psi,\tilde{\psi})<\epsilon$.
By Lemma \ref{lemimportant} we can choose a path $\psi_t$ with $\psi_0=\psi$ and $\psi_1=\tilde{\psi}$ such that $\min\limits_M\alpha(X_t^{\psi})>-\epsilon$.
Define
$$\tilde{\phi}_t:=\phi_t\circ\psi^{-1}\circ \psi_t.$$
Then $\tilde{\phi}_0=\phi$ and $\tilde{\phi}_1=\tilde{\psi}$.
Moreover
\begin{align*}
\min\limits_M\alpha(X_t^{\tilde{\phi}})&\geq \min\limits_M(\phi_t\psi^{-1})^{\ast}\alpha(X_t^{\psi})+\min\limits_M\alpha(X_t^{\phi})\\
&\geq \min(0,\min\limits_M\alpha_0(X_t^{\psi}))+\min\limits_M\alpha(X_t^{\phi})>0.
\end{align*}
Hence $\tilde{\phi}_t$ is a positive path from $\phi$ to $\tilde{\psi}$, i.e. $\tilde{\psi}\in I^+(\phi)$.
It follows that $I^+(\phi)$ is open with respect to $d_{\alpha_0}$ and thus with respect to any $d_{\alpha}$ since these metrics are equivalent.
\end{proof}

\begin{proof}[Proof of Theorem \ref{thm1}]
\begin{itemize}
\item[(i)]
By definition $\phi\llcurly \psi\Rightarrow \tau_{\alpha}(\phi,\psi)>0$.
If $\tau_{\alpha}(\phi,\psi)>0$ there exists a path $\phi_t$ between $\phi$ and $\psi$ with $\epsilon:=\integ{0}{1}{\min\limits_M\alpha(X_t^{\phi})}{t}>0$.
Lemma \ref{lemimportant} implies the existence of a positive path between $\phi$ and $\psi$ by setting e.g. $\delta=\frac{\epsilon}{2}$.
\item[(ii)]
We show that $\tau_{\alpha}$ is continuous with respect to the interval topology.
Then the claim follows by Proposition \ref{propinterval}.
Let $(\phi,\psi)\in \tau_{\alpha}^{-1}((a,b))$, where $(a,b)\subset [0,\infty]$ is an open interval or $(a,b)=[0,b)$ or $(a,b)=(a,\infty]$.
Denote by $\phi_t^{\alpha}$ the Reeb flow of $\alpha$.

\textbf{Claim 1: }There exists $r>0$ such that $\tau_{\alpha}(\phi_t^{\alpha}\phi, \phi_s^{\alpha}\psi) \in (a,b)$ for all $s,t\in [-r,r]$.

Suppose not.

\textbf{1. Case:} Assume $b<\infty$ and for arbitrarily small $r$ there exist $t_0,s_0$ and a path $\psi_t$ with $\psi_0=\phi_{t_0}^{\alpha}\phi$ and $\psi_1=\phi_{s_0}^{\alpha}\psi$ such that $\integ{0}{1}{\min\limits_M\alpha(X_t^{\psi})}{t}\geq b$.

Due to Lemma \ref{lemimportant} one can assume that $\min\limits_M\alpha(X_t^{\psi})>b-\delta$ for $\delta$ arbitrarily small.
Let $\tau(t):=(t-1)t_0-ts_0$ and $\tilde{\psi}_t:=\phi_{\tau(t)}^{\alpha}\psi_t$.
Then $\tilde{\psi}_0=\phi$, $\tilde{\psi}_1=\psi$ and since $\phi_{\tau(t)}^{\alpha}$ is a strict contactomorphism
$$\min\limits_M\alpha(X_t^{\tilde{\psi}})\geq \min\limits_M\alpha(X_t^{\psi})+(s_0-t_0)> b-2r-\delta.$$
Since we can choose $r$ and $\delta$ arbitrarily small this contradicts $(\phi,\psi)\in \tau_{\alpha}^{-1}((a,b))$.

\textbf{2. Case: }Assume $a>0$ and for arbitrarily small $r$ there exist $t_0,s_0$ with $\tau_{\alpha}(\phi_{t_0}^{\alpha}\phi,\phi_{s_0}^{\alpha}\psi)\leq a$.

There exists a path $\psi_t$ with $\psi_0 =\phi$, $\psi_1=\psi$ and $\min\limits_M\alpha(X_t^{\psi})>a$ for all $t$.
Using the same argument like in Case 1 for small $\delta$ one can construct a path $\tilde{\psi}_t$ between $\phi_{t_0}^{\alpha}\phi$ and $\phi_{s_0}^{\alpha}\psi$ with $\integ{0}{1}{\min\limits_M\alpha(X_t^{\tilde{\psi}})}{t}>a$.

One can assume that always one of the two cases holds.
If $b=\infty$ every $t,s$ satisfies $\tau_{\alpha}(\phi_{t}^{\alpha}\phi,\phi_{s}^{\alpha}\psi)\leq b$.
Then either $a=0$, i.e. $\tau_{\alpha}^{-1}((a,b))=\mathrm{Cont}_0(M)\times\mathrm{Cont}_0(M)$ or the second case above holds.
Similarly if $a=0$ every $t,s$ satisfies $\tau_{\alpha}(\phi_{t}^{\alpha}\phi,\psi_{s}^{\alpha}\phi)\geq a$,
Then either $b=\infty$ or the first case above holds.

This shows Claim 1.

\textbf{Claim 2:}Choose $r>0$ like in  Claim 1.
Then $(I^+(\phi_{-r}^{\alpha}\phi)\cap I^-(\phi_{r}^{\alpha}\phi))\times (I^+(\phi_{-r}^{\alpha}\psi)\cap I^-(\phi_{r}^{\alpha}\psi))\subset \tau_{\alpha}^{-1}(a,b)$.

W.l.o.g. assume $a>0$ and $b<\infty$.
Take $\tilde{\phi} \in I^+(\phi_{-r}^{\alpha}\phi)\cap I^-(\phi_{r}^{\alpha}\phi) $ and $\tilde{\psi}\in I^+(\phi_{-r}^{\alpha}\psi)\cap I^-(\phi_{r}^{\alpha}\psi)$.
There exists a positive path $\psi_t$ with $\psi_0=\tilde{\phi}$, $\psi_{\frac{1}{3}}= \phi_{r}^{\alpha}\phi$, $\psi_{\frac{2}{3}}= \phi_{-r}^{\alpha}\psi$ and $\psi_1=\tilde{\psi}$ such that
$$\integ{0}{1}{\min\limits_M\alpha(X_t^{\psi})}{t}> a.$$
Here we used that $\integ{0}{1}{\min\limits_M\alpha(X_t^{\psi})}{t}$ is invariant under smooth orientation preserving re-parametrisation of $\psi_t$ with respect to $t$, i.e. one can choose $\psi_t$ such that 
$$\integ{\frac{1}{3}}{\frac{2}{3}}{\min\limits_M\alpha(X_t^{\psi})}{t}=\tau_{\alpha}(\phi_{r}^{\alpha}\phi,\phi_{-r}^{\alpha}\psi)-\delta > a$$
for $\delta$ arbitrarily small.
On the other hand, the existence of a path $\psi_t$ between $\tilde{\phi}$ and $\tilde{\psi}$ with $\integ{0}{1}{\min\limits_M\alpha(X_t^{\psi})}{t}\geq b$ would imply that $\tau_{\alpha}(\phi_{-r}^{\alpha}\phi,\phi_{r}^{\alpha}\psi)>b$.
This shows Claim 2, in particular, $\tau_{\alpha}$ is continuous with respect to the interval topology.
\item[(iii)]
Assume $id_M\preccurlyeq\psi\preccurlyeq\phi$.
Take non-negative paths $\psi_t$ from $id_M$ to $\psi$ and $\tilde{\psi}_t$ from $id_M$ to $\phi_2$.
Following \cite{Egor} take smooth functions $\abb{\tau_1}{[0,1]}{[0,1]}$ and $\abb{\tau_2}{[0,1]}{[0,1]}$ with $\mathrm{supp}(\tau_1')\subset [0,\frac{1}{2}]$, $\mathrm{supp}(\tau_2')\subset [\frac{1}{2},1]$ and $\tau_i'\geq 0$ such that $\tau_1(0)=0=\tau_2(0)$ and $\tau_1(1)=1=\tau_2(1)$.
Then $\hat{\psi}_t:=\psi_{\tau_1(t)}\tilde{\psi}_{\tau_2(t)}$ defines a smooth path between $id_M$ and $\psi\phi$ with
$$\integ{0}{1}{\min\limits_M\alpha(X_t^{\hat{\psi}_t})}{t}\geq \integ{0}{1}{\min\limits_M\alpha(X_t^{\psi})}{t}+\integ{0}{1}{\min\limits_M\alpha(X_t^{\tilde{\psi}})}{t}.$$
It follows that 
$$\tau_{\alpha}(id_M,\psi\phi)\geq \tau_{\alpha}(id_M,\psi)+\tau_{\alpha}(id_M,\phi).$$
Then property (iii) follows from the right invariance of $\tau_{\alpha}$.
\item[(iv)]
Analogous to property (iv) in \cite{Egor} .
\end{itemize}
\end{proof}

\begin{proof}[Proof of Corollary \ref{thmq18}]
Assume that $|\phi_{-t}^{\alpha}|_{\alpha}<t$, where $t>0$.
Then there exists a path $\psi_s$ with $\psi_0=id_M$ and $\psi_1=\phi_{-t}^{\alpha}$ such that $\integ{0}{1}{\max\limits_M|\alpha(X_s^{\psi})|}{s}<t$.
By Lemma \ref{lemimportant} we can assume that $\min\limits_M\alpha(X_s^{\psi})>-t$ for all $s$.
Define $\tilde{\psi}_s:=\phi_{ts}^{\alpha}\psi_s$.
The path $\tilde{\psi}_s$ is a loop with $\tilde{\psi}_0=id_M=\tilde{\psi}_1$.
Moreover since $\phi_{ts}^{\alpha}$ is a strict contactomorphism one has
\begin{align*}
\min\limits_M\alpha(X_s^{\tilde{\psi}})&\geq \min\limits_M\alpha(X_s^{\psi})+t>0.
\end{align*}
This contradicts the orderability of $(M,\xi)$.
The corollary follows from the fact that $|\phi_t^{\alpha}|_{\alpha}=|\phi_{-t}^{\alpha}|_{\alpha}$.
\end{proof}

\begin{proof}[Proof of Corollary \ref{thm2}]
If $\mathrm{Cont}_0(M)$ is not orderable one can choose a positive loop $\phi_t$ (\cite[Proposition 2.1.B]{Eliashberg}).
Iterating this loop, we get $\tau_{\alpha}(\phi_0,\phi_0)=\infty$.
Let $\mathrm{Cont}_0(M)$ be orderable, i.e. there is no non-negative loop in $\mathrm{Cont}_0(M)$.
Take $\phi \preccurlyeq \psi$.
Assume that $\tau_{\alpha}(\phi,\psi)=\infty$.
Hence for any $c>0$ there exists a positive path $\phi_t$ from $\phi$ to $\psi$ with $\integ{0}{1}{\min\limits_M\alpha(X_t^{\phi})}{t}>c$.
Due to Lemma \ref{lemimportant} one can assume that $\min\limits_M\alpha(X_t^{\phi})>c$ for all $t$.
Fix a path $\psi_t$ from $id_M$ to $\phi\circ \psi^{-1}$.
Define
$$\tilde{\phi}_t:=\psi_t\circ \phi_t.$$
Then $\tilde{\phi}_0=\phi=\tilde{\phi}_1$.
Let $\rho_t$ be the family of positive functions with $\psi_t^{\ast}\alpha=\rho_t\alpha$ and $\alpha_0:=\rho_{\min} \alpha$, where $\rho_{\min}:=\min\limits_{[0,1]\times M}\rho_t$.
Then
\begin{align*}
\min\limits_M\alpha(X_t^{\tilde{\phi}})&\geq \min\limits_M\psi_t^{\ast}\alpha(X_t^{\phi})+\min\limits_M\alpha(X_t^{\psi})\\
&\geq \min\limits_M\alpha_0(X_t^{\phi})+\min\limits_M\alpha(X_t^{\psi}).
\end{align*}
Since $\tau_{\alpha}(\phi,\psi)=\infty$ implies that $\tau_{\alpha_0}(\phi,\psi)=\infty$ and $c$ was arbitrarily big, one can choose $\phi_t$ such that
$$\min\limits_M\alpha(X_t^{\tilde{\phi}})>0.$$
Then $\tilde{\phi}$ is a positive loop.
This contradicts the orderability of $(M,\xi)$.

Similarly if $\mathrm{Cont}_0(M)$ is orderable, consider the Reeb flow $\phi_t^{\alpha}$ of $\alpha$.
Assume that $\tau_{\alpha}(\phi,\phi_t^{\alpha}\phi)>t$, where $t\geq 0$.
Take a path $\psi_s$ from $\phi$ to $\phi_t^{\alpha}\phi$ with $\min\limits_M\alpha(X_s^{\phi})>t$ for all $s\in[0,1]$.
Then $\phi_{-ts}^{\alpha}\psi_s$ defines a positive loop contradicting the orderability of $(M,\xi)$.
\end{proof}

\section{Spherical co-tangent bundles and spectral invariants}\label{sec4}

To prove Theorem \ref{thm3} we will use spectral invariants on jet-spaces as defined for instance in \cite{Colin}, \cite{Chernov10} or \cite{Chernov19}.
We will follow the conventions in \cite{Zapolsky}.
Let $N$ be a closed manifold.
Its $1$-jet space is given by
$$J^1N:=T^{\ast}N\times \real{}.$$
The space $J^1N$ canonically carries the contact structure defined by the $1$-form $\alpha_0:=\lambda+dt$, where $\lambda$ denotes the canonical Liouville-form on $T^{\ast}M$ and $t$ the projection to the $\real{}$-component.
An important example of Legendrian submanifolds are given by $1$-jets of functions.
For $\abb{f}{N}{\real{}}$ define its $1$-jet by
$$j^1f:=\{(p,-df_p,f(p))\in J^1N|p\in N\}.$$
In particular the $0$-section of $J^1N$ as the $1$-jet of the $0$-function is a Legendrian submanifold.
Note that all $1$-jets are Legendrian isotopic to the $0$-section.
On the other hand not all Legendrians in this isotopy class are $1$-jets of some function.
An easy way to describe them is given by generating functions:

For some $m\in \mathbb{N}$ consider a smooth function $\abb{S}{N\times \real{m}}{\real{}}$.
Look at the set $\Sigma_S:=\{(p,e)\in N\times\real{m}|d_eS=0\}$, where $d_eS$ denotes the fibre differential in the $\real{m}$ direction.
If $0$ is a regular value of $d_eS$, $\Sigma_S$ is a submanifold of $N\times\real{m}$.
Moreover the map $i_S(p,e):=(p,-d_pS(p,e),S(p,e))$ immerses $\Sigma_S$ into $J^1(N)$ as a Legendrian submanifold.
Here $d_pS$ denotes the differential in the $N$ direction which is well defined (independent of the choice of a horizontal distribution) on $\Sigma_S$.
The function $S$ is called a \textit{generating function} for the Legendrian $i_S(\Sigma_S)$.
We say that $S$ is \textit{quadratic at infinity} if $S$ is of the form
$$S(p,e)=f(p,e)+Q(e),$$
where $f$ is compactly supported and $Q$ is a non-degenerate quadratic form.
For Legendrians isotopic to the $0$-section one can show

\begin{theo}[\cite{Chekanov}]
Let $L\subset J^1N$ be a Legendrian submanifold Legendrian isotopic to the $0$-section.
Then there exists a quadratic at infinity generating function $S$ generating $L$.
Moreover if $L_t$ is a smooth isotopy of Legendrians isotopic to $0$, there exists a smooth family of quadratic at infinity generating functions $S_t$ such that $S_t$ generates $L_t$.
\end{theo}

Generating functions can be used to define spectral invariants for Legendrian submanifolds.
Let $S$ be a quadratic at infinity generating function for a Legendrian submanifold $L$.
Since $S=f+Q$ and the function $f$ has compact support, there exist $a_0,b_0\in \real{}$ such that $H_{\ast}(\{S<b\},\{S<a\},\mathbb{Z}_2)\cong H_{\ast}(\{S<b_0\},\{S<a_0\},\mathbb{Z}_2)$ for all $a\leq a_0<b_0\leq b$.
Let $q$ be the (negative) index of $Q$.
There exists a natural graded isomorphism
$$\abb{i}{H_{\ast}(\{S<b\},\{S<a\},\mathbb{Z}_2)}{H_{\ast-q}(N,\mathbb{Z}_2)}$$
that is independent of $a\leq a_0$ and $b\geq b_0$ (see \cite{Zapolsky}).
Thus for $a\leq a_0$ and any $b>a$ there is a natural inclusion $\abb{i^b}{H_{\ast}(\{S<b\},\{S<a\},\mathbb{Z}_2)}{H_{\ast-q}(N,\mathbb{Z}_2)}$.
For $A\in H_{\ast}(N,\mathbb{Z}_2)$ define
$$l(L,A):=\inf \{b\in\real{}| A\in \mathrm{im}(i^b)\}.$$
Due to the Viterbo-Theret uniqueness theorem \cite{Viterbo},\cite{Theret} the spectral invariant $l(L,A)$ is independent of the choice of a generating function.
We will use the following Lemma from Zapolsky, for further properties of the spectral invariants see e.g. \cite{Zapolsky}.

\begin{lemma}[\cite{Zapolsky}]\label{lemzap}
Let $L\subset J^1N$ be a Legendrian submanifold isotopic to $0$ and $\phi_t$ a path of contactomorphisms such that $\phi_0=id_M$ and $\phi_1(0)=L$.
Then for any $A\in H_{\ast}(N)$
$$\integ{0}{1}{\min\limits_{\phi_t(0)}\alpha_0(X_t^{\phi})}{t}\leq l(L,A)\leq \integ{0}{1}{\max\limits_{\phi_t(0)}\alpha_0(X_t^{\phi})}{t}.$$
\end{lemma}

\begin{proof}[Proof of Theorem \ref{thm3}]
Assume that $M=ST^{\ast}N$, where $N$ is smoothly covered by an open subset of $\real{n}$.
W.l.o.g this open subset contains the origin.
Consider a positive path of contactomorphisms $\phi_t$ in $\mathrm{Cont}_0(M)$.
Note that $M$ is covered by $ST^{\ast}\real{n}$ such that the projection map is a local contactomorphism.
Hence $\phi_t$ lifts to a positive path $\tilde{\phi}_t$ on $ST^{\ast}\real{n}$.
Since $\tau_{\alpha}$ and $d_{\alpha}$ are right invariant we can assume that $\phi_0=id_M$.
As in \cite{Chernov10} we use the hodograph transform to get a contactomorphism from  $ST^{\ast}\real{n}$ to $J^1(S^{n-1})$ such that the fibre over the origin is mapped to the $0$-section in $J^1(S^{n-1})$.
In particular the spectral invariants can be defined for Legendrians in $ST^{\ast}\real{n}$ isotopic to the fibres.
Denote by $\tilde{\alpha}$ the contact form on $M$ induced by the standard contact form on $J^1(S^{n-1})$ and by $L$ a Legendrian that lifts to the fibre over the origin.
Then using Lemma \ref{lemzap} we get
\begin{align*}
\integ{0}{1}{\min\limits_M\tilde{\alpha}(X_t^{\phi})}{t}&\leq \integ{0}{1}{\min\limits_{\phi_t(L)}\tilde{\alpha}(X_t^{\phi})}{t}\leq l(\phi_1(L),A)\\
&\leq \integ{0}{1}{\max\limits_{\phi_t(L)}\tilde{\alpha}(X_t^{\phi})}{t}\leq \integ{0}{1}{\max\limits_{M}|\tilde{\alpha}(X_t^{\phi})|}{t}.
\end{align*}
Note that the spectral invariant $l(\phi_1(L),A)$ is independent of the choice of path between $id_M$ and $\phi_1$.
Taking the supremum over all paths from $id_M$ to $\phi_1$ on the left hand side and the infimum over all paths on the right hand side shows
$$\tau_{\tilde{\alpha}}(\phi_0,\phi_1)\leq d_{\tilde{\alpha}}(\phi_0,\phi_1).$$
For $\alpha=\rho \tilde{\alpha}$ one gets
$$\tau_{\alpha}(\phi_0,\phi_1)\leq \frac{\max\rho}{\min\rho} d_{\alpha}(\phi_0,\phi_1).$$
\end{proof}

\section{A Lorentzian distance function for Legendrian isotopy classes}\label{sleg}

Let $(M,\xi)$ be any co-orientable contact manifold (not necessarily closed).
For a closed Legendrian denote by $\mathcal{L}$ its Legendrian isotopy class.
Given a parametrisation $\abb{l_t}{L}{L_t}$ of a Legendrian isotopy denote by $X_t^l$ the section of $TM|_{L_t}$ defined by
$$X_t^l(l_t(p)):=\frac{d}{ds}|_{s=t}l_s(p).$$
A Legendrian isotopy $L_t$ is called \textit{ positive (non-negative)} if there is a parametrisation $\abb{l_t}{L}{L_t}$ such that
$$\alpha(X_t^l)>(\geq)0.$$
Given a positive (non-negative) isotopy $L_t$ of closed Legendrians, there exists a compactly supported non-negative path of contactomorphisms $\phi_t$ such that $L_t=\phi_{f(t)}(L)$ for a non-decreasing function $f$ (see \cite[Proposition 4.1]{Chernov16}).
Analogous to the relations on $\mathrm{Cont}_0(M)$ positivity defines relations $\llcurly$ and $\preccurlyeq$ on $\mathcal{L}$.
The isotopy class $\mathcal{L}$ is called \textit{orderable} if $\preccurlyeq$ is a partial order.

There also holds a  version of Lemma \ref{lemimportant} for Legendrian isotopy classes.

\begin{lemma}\label{lemimportant2}
Let $\abb{l_t}{L}{L_t}$ be a Legendrian isotopy with 
$$\integ{0}{1}{\min\limits_{L_t}\alpha(X_t^{l})}{t}=\epsilon.$$
Assume that $l_t=\phi_t|_{L_0}$ for some compactly supported path of contactomorphisms $\phi_t$.
Then for all $\delta>0$ there exists $\abb{\tilde{l}_t}{L_0}{\tilde{L}_t}$ with $\tilde{L}_1=L_1$ such that
$$\min\limits_{L_t}\alpha(X_t^{\tilde{l}})\in (\epsilon-\delta,\epsilon+\delta).$$
\end{lemma}

\begin{proof}
Take a compactly supported path of contactomorphisms such that $\phi_t|_{L_0}=l_t$.
Let $\abb{\tau}{[0,1]}{\real{}}$ with $\tau(0)=0$ and
$$\tau'(t):=-\min\limits_{L_t}\alpha(X_t^{l})+\epsilon.$$
Then $\psi_t:=\phi_{\tau(t)}^{\alpha}\phi_t$ satisfies $\psi_0(L_0)=L_0$ and $\psi_1(L_0)=L_1$.
Moreover since $\phi_{\tau(t)}^{\alpha}$ is a strict contactomorphism
$$\min\limits_{\psi_t(L_0)}\alpha(X_t^{\psi})=\min\limits_{L_t}\alpha(X_t^{l})+\tau'(t)=\epsilon.$$
Approximating $\tau$ by smooth functions with fixed endpoints finally gives $\tilde{l_t}$ with the desired properties.
\end{proof}

For closed $M$ Rosen and Zhang in \cite{Rosen} defined a Chekanov-type metric on the orbit space of a subset $M$ under the action of $\mathrm{Cont}_0(M)$.
For Legendrians $L_0,L_1\in\mathcal{L}$ it is given by
$$d_{\alpha}(L_0,L_1):=\inf\{|\phi|_{\alpha}|\phi(L_0)=L_1\}.$$

The map $d_{\alpha}$ has the following properties:

\begin{theo}[\cite{Rosen}]
The map $d_{\alpha}$ satisfies
\begin{itemize}
\item[(i)]$d_{\alpha}(L,L)=0$
\item[(ii)]$d_{\alpha}(L_0,L_1)=d_{\alpha}(L_1,L_0)$
\item[(iii)]$d_{\alpha}(L_0,L_2)\leq d_{\alpha}(L_0,L_1)+d_{\alpha}(L_1,L_2)$
\item[(iv)]For $\phi\in \mathrm{Cont}_0(M)$ generated by a contact Hamiltonian $F$ there exist constants $C^-(\phi,F)$ and $C^+(\phi,F)$ such that
$$C^-(\phi,F)d_{\alpha}(L_0,L_1)\leq d_{\alpha}(\phi(L_0),\phi(L_1))\leq C^+(\phi,F)d_{\alpha}(L_0,L_1).$$
\end{itemize}
\end{theo}

\begin{remark}
Due to the non-conjugation invariance of the Hofer-Shelukhin norm this metric fails to be left invariant, i.e. in general $d_{\alpha}(L_0,L_1)\neq d_{\alpha}(\phi(L_0),\phi(L_1))$.
\end{remark}

For closed Legendrians \cite[Conjecture 1.10]{Rosen} states that $d_{\alpha}$ is always non-degenerate.
The conjecture was proven in \cite[Corollary 3.5]{Usher} for hypertight Legendrians.
A simple proof of \cite[Conjecture 1.10]{Rosen} can be given for orderable Legendrian isotopy classes.

\begin{theorem}\label{thmnd}
Let $L$ be a closed Legendrian such that $\mathcal{L}$ is orderable.
Then $d_{\alpha}$ is non-degenerate.
\end{theorem}

\begin{proof}
By \cite[Theorem 1.9]{Rosen} $d_{\alpha}$ is either non-degenerate or vanishes identically.
Thus it suffices to show that $d_{\alpha}(L_0,L_1)>0$ for two Legendrians $L_0,L_1\in\mathcal{L}$.
Look at the inverse of the time-$1$ map of the Reeb-flow $\phi_{-1}^{\alpha}$.
Suppose $d_{\alpha}(L,\phi^{\alpha}_{-1}(L))=0$ for some $L\in\mathcal{L}$.
Then for any $\epsilon>0$ there exists a compactly supported $\phi_t$ with $\phi_0=id_M$ and $\phi_1(L)=\phi_{-1}^{\alpha}(L)$ such that
$$\integ{0}{1}{\max\limits_{M}|\alpha(X_t^{\phi})|}{t}<\epsilon.$$
Look at $\psi_t:=\phi_t^{\alpha}\phi_t$.
Then $\psi_0(L)=\psi_1(L)=L$.
Moreover since $\phi_t^{\alpha}$ is a strict contactomorphism 
$$\integ{0}{1}{\min\limits_{M}\alpha(X_t^{\psi})}{t}\geq 1+ \integ{0}{1}{\min\limits_{M}\alpha(X_t^{\phi})}{t}\geq 1-\epsilon>0$$
for $\epsilon<1$.
Lemma \ref{lemimportant2} implies that there exists a positive loop in $\mathcal{L}$.
\end{proof}

\begin{remark}
In \cite[Example 1.12]{Colin2} the authors give an example of a hypertight Legendrian contained in a (non-contractible) positive loop, i.e. there are non-orderable hypertight Legendrians.
On the other hand let $M$ be a closed manifold such that its universal cover is not compact and $\pi_k(M)\neq 0$ for some $k>1$.
Let $\abb{i}{M}{\Lambda_0 M}$ be the map that maps a point $p$ in $M$ to the constant loop through $p$ in the component of contractible loops of the free loop space $\Lambda_0 M$.
In this case the induced map in singular homology is not surjective.
It follows from the proof of \cite[Theorem 4.1]{Viterbo2} that $ST^{\ast}M$ does not admit a contact form without contractible Reeb orbits (see also \cite{Abbondandolo} for the isomorphism between the Floer homology of co-tangent bundles and the singular homology of the free loop space).
Thus $ST^{\ast}M$ is not hypertight but as shown in \cite{Chernov102} the isotopy class of the fibre is orderable.
This shows that Theorem \ref{thmnd} as well as \cite[Corollary 3.5]{Usher} are not covered by the other result.
\end{remark}

\begin{remark}
Since $M$ is closed, the metrics induced by different contact forms are equivalent.
\end{remark}

\begin{cor}\label{cor}
Let $L$ be a closed Legendrian such that $\mathcal{L}$ is orderable.
Then the interval topology on $\mathcal{L}$ is coarser than the topology induced by $d_{\alpha}$, in particular, the sets $I^{\pm}(L)$ are open in this topology.
\end{cor}

\begin{proof}
The proof works analogous to the one of Proposition \ref{propinterval}.
\end{proof}

\begin{remark}
Using compactly supported contactomorphisms it is possible to define $d_{\alpha}$ for non closed contact manifolds.
The proofs of Theorem \ref{thmnd} and Corollary \ref{cor} should also work in this case.
However it is not clear if the metrics defined by different contact forms are equivalent or induce the same topology on $\mathcal{L}$.
\end{remark}

For Legendrian isotopy classes it is possible to prove a version of Conjecture \ref{con1}.
Thus any reasonable Lorentzian distance function is not invariant under the action of $\mathrm{Cont}_0(M)$.

\begin{prop}\label{propnonex}
Let $\abb{\tau}{\mathcal{L}\times\mathcal{L}}{[0,\infty]}$ be a map satisfying $\tau(L_0,L_1)>0$ iff $L_0\llcurly L_1$ and $\tau(L_0,L_{2})\geq \tau(L_0,L_1)+\tau(L_1,L_2)$ for $L_0\preccurlyeq L_1\preccurlyeq L_2$.
Assume that $\tau(\phi L_0,\phi L_1)=\tau(L_0,L_1)$ for any $\phi\in \mathrm{Cont}_0(M)$. 
Then 
\begin{align*}
&\tau(L_0,L_1)=\left\lbrace \begin{array}{cc}
\infty, &\text{if } L_0\llcurly L_1 \\ 
0, & \text{otherwise}
\end{array} \right..
\end{align*}
\end{prop}

\begin{proof}
Let $L$ be a closed Legendrian and assume there exists such a function $\tau$ on its Legendrian isotopy class $\mathcal{L}$.
Let $L_2\in \mathcal{L}$ with $L\llcurly L_2$. 
Choose a positive path $\phi_t$ with $\phi_0=L$ and $\phi_1=L_2$.
Then for $0<t_0<t_1<1$ small enough and $L_0:=\phi_{t_0}(L), L_1:=\phi_{t_1}(L)$ one can assume that $L_0 \cap L= L_1\cap L =L_0\cap L_1=\emptyset$.

\textbf{Claim:} There exists $\psi \in \mathrm{Cont}_0(M)$ with $\psi(L_0)=L_1$ and $\psi(L)=L$.

Take an open neighbourhood $U$ of $L$ such that $\phi_t(L)\cap U=\emptyset$ for all $t\in [t_0,t_1]$.
Let $V$ be an open neighbourhood of $L$ with $\overline{V}\subset U$.
Let $\abb{F}{[t_0,t_1]\times M}{\real{}}$ be a smooth function with $F(t,p)=\alpha(X_t^{\phi\phi_{t_0}^{-1}})(p)$ for $p\in M\setminus U$ and $F(t,p)=0$ for $p\in V$.
Define $\psi$ to be the flow of the time-dependent contact Hamiltonian vector field of $F$ at time $t_1$.
Then $\psi=id_M$ on $V$ and since for $t\in [t_0,t_1]$ we have $\phi_t(L)\cap U=\emptyset$, $\psi$ coincides with $\phi_{t_1}\phi_{t_0}^{-1}$ around $L_0$.

It follows
$$\tau(L,L_0)=\tau(\psi(L),\psi(L_0))=\tau(L,L_1)\geq \tau(L,L_0)+\tau(L_0,L_1).$$
Since $\tau(L_0,L_1)>0$ this implies $\tau(L,L_0)=\infty$.
Due to the reverse triangle inequality $\tau(L,\cdot)$ is strictly increasing along positive paths.
Hence $\tau(L,L_2)=\infty$.
\end{proof}

We define a Lorentzian distance function on $\mathcal{L}$ that is not $\mathrm{Cont}_0(M)$-invariant by 

\begin{definition}
Define
\begin{align*}
\tau_{\alpha}(L_0,L_1):=\left\lbrace \begin{array}{cc}
\sup\left\lbrace\integ{0}{1}{\min\limits_{L_t}\alpha(X_t^{l})}{t} \right\rbrace, & \text{if } L_0\preccurlyeq L_1 \\ 
0, & \text{otherwise}
\end{array} \right.
\end{align*}
Here the supremum is taken over all non-negative Legendrian isotopies $l_t$ with $l_0(L_0)=L_0$ and $l_1(L_0)=L_1$ such that $l_t=\phi_t|_{L_0}$ for some compactly supported non-negative path of contactomorphisms $\phi_t$.
\end{definition}

Using Lemma \ref{lemimportant2} one can analogously to the case of the Lorentzian distance on $\mathrm{Cont}_0(M)$ prove the following.

\begin{theorem}\label{thm12}
The map $\tau_\alpha$ satisfies
\begin{itemize}
\item[(i)]$\tau_{\alpha}(L_0,L_1)>0\Leftrightarrow L_0\llcurly L_1$.
\item[(ii)]$\tau_{\alpha}$ is continuous with respect to the interval topology.
\item[(iii)]$\tau_{\alpha}(L_0,L_2)\geq \tau_{\alpha}(L_0,L_1)+\tau_{\alpha}(L_1,L_2)$ for $L_0\preccurlyeq L_1\preccurlyeq L_2$.
\end{itemize}
If $M$ is closed and $\mathcal{L}$ orderable then 
\begin{itemize}
\item[(iv)]$\tau_{\alpha}$ is continuous with respect to the topology induced by $d_{\alpha}$.
\end{itemize}
\end{theorem}

\begin{theorem}
The Legendrian isotopy class $\mathcal{L}$ is orderable if and only if $\tau_{\alpha}(L_0,L_1)<\infty$ for all $L_0,L_1\in \mathcal{L}$.
In this case for $t\geq 0$
$$\tau_{\alpha}(L,\phi_t^{\alpha}(L))=t.$$
\end{theorem}

Since Lemma \ref{lemzap} is formulated in terms of Legendrians, one also has

\begin{theorem}
Assume that $(M,\xi)\cong (ST^{\ast}N,\xi_{st})$, where $N$ is smoothly covered by an open subset of $\real{n}$.
Then there exists a constant $C_{\alpha}$ depending on $\alpha$ such that for all $L_0,L_1$ isotopic to the fibres one has
$$\tau_{\alpha}(L_0,L_1)\leq C_{\alpha} d_{\alpha}(L_0,L_1).$$
\end{theorem}

\section{A metric for globally hyperbolic spacetimes}

In \cite{Low} Low constructed the \textit{space of null geodesics} $\mathcal{N}_g$ of a Lorentzian manifold $(N,g)$.
He observed that for globally hyperbolic $(N,g)$ the space of null geodesics naturally carries the structure of a smooth contact manifold.
Moreover given a Cauchy hypersurface $\Sigma$ he constructed a contactomorphism $\rho_{\Sigma}$ from $\mathcal{N}_g$ to $ST^{\ast}\Sigma$ equipped with its standard contact structure.
Given a point $p\in N$ its \textit{sky} $S(p)$ is the set of null geodesics through the point $p$.
Denote by $S(N)$ the set of all skies.
The set $S(p)$ is always a Legendrian submanifold of $\mathcal{N}_g$ and is mapped by $\rho_N$ to the isotopy class of the fibres in $ST^{\ast}N$ (\cite{Low}).

Chernov and Nemirovski \cite[Proposition 4.5]{Chernov19} proved that if the isotopy class of the fibre is orderable, then the map $p\mapsto S(p)$ bijectively maps $N$ to $S(N)$ so that the natural orders on both sets coincide.
Hence the maps $\tau_{\alpha}$ and $d_{\alpha}$ restrict to $S(N)$ and induce a Lorentzian distance and a metric on $N$.

\begin{theorem}\label{thmglobhyp}
In the case described above the metric $d_{\alpha}^N:=d_{\alpha}|_{S(N)\times S(N)}$ induces the manifold topology on $N$.
\end{theorem}

\begin{proof}
Since the isotopy class of the fibres in $ST^{\ast}\Sigma$ is orderable Corollary \ref{cor} implies that the interval topology on $S(N)$ is open with respect to the topology induced by $d_{\alpha}$.
Thus \cite[Corollary 4.6]{Chernov19} implies that the manifold topology on $N$ is coarser than the topology induced by $d_{\alpha}^N$.

Due to the Bernal-S{\'a}nchez theorem (see e.g. \cite{Minguzzi2}) one can assume that $N=\real{}\times \Sigma$, where $\{t\}\times \Sigma$ is a Cauchy hypersurface for every $t$.
Let $(t_0,p)\in N$ and $\epsilon>0$.
W.l.o.g. assume $t_0=0$.
Since $(N,g)$ is strongly causal it suffices to show that there exists $\delta>0$ such that $d_{\alpha}(S((0,p)),S((t,q)))<\epsilon$ for any $(t,q)\in I^+((-\delta,p))\cap I^-((\delta,p))$.
Let $\abb{\rho_t}{\mathcal{N}_g}{ST^{\ast}(\{t\}\times \Sigma)}$ be the natural contactomorphism described in \cite{Low}.
Define $\phi_t:=\rho_0\circ\rho_t^{-1}$.
Using the natural identification one has $\phi_t\in \mathrm{Cont}_0(ST^{\ast}(\{0\}\times \Sigma))$.
Denote by $F_q$ the fibre over the point $q\in \Sigma$.
Then by definition $S(t,q)=\phi_t(F_q)$.
In particular for a curve of the form $(f(t),\gamma(t))$ one has $S(f(t),\gamma(t))=\phi_{f(t)}(F_{\gamma(t)})$.
Choose a parametrisation $\abb{l_t}{S^n}{F_{\gamma(t)}}$.
Then for $w\in S((f(t),\gamma(t))$  and $u\in S^{n}$ with $w=\phi_{f(t)}(l_{t}(u))$
\begin{align*}
\alpha_{w}\left(\frac{d}{dt}\phi_{f(t)}(l_t(u))\right)&=\alpha_{w}\left(d\phi_{f(t)}\left(\frac{d}{dt}l_t(u)\right)+f'(t) X^{\phi}_{f(t)}(w)\right)\\
&=(\phi_{f(t)}^{\ast}\alpha)_{l_{t}(u)}(\hat{\gamma}'(t))+f'(t)\alpha_w( X^{\phi}_{f(t)}(w)).
\end{align*}
Here $\hat{\gamma}'(t)$ denotes any vector $v\in T_{l_{t}(u)}ST^{\ast}\Sigma$ with $d\pi v=\gamma'(t)$.

Define $C:=\max\limits_{t\in [0,1]}\max\limits_{\phi_t(F_p)}|\alpha(X_t^{\phi})|$.
It follows that for $0<\delta<1 $
$$d_{\alpha}(S((0,p)),S((\delta,p)))\leq \integ{0}{\delta}{\max\limits_{\phi_t(F_p)}|\alpha(X_t^{\phi})|}{t}\leq C\delta$$
and 
$$d_{\alpha}(S((0,p)),S((-\delta,p)))\leq \integ{-\delta}{0}{\max\limits_{\phi_t(F_p)}|\alpha(X_t^{\phi})|}{t}\leq C\delta.$$
Let $(s,q)\in I^+((-\delta,p))\cap I^-((\delta,p))$.
Since the metrics induced by different contact forms are equivalent one can assume that $\alpha$ is induced by a Riemannian metric $h$, i.e. $\alpha_{[v]}(w)=v(d\pi(w))$, where $v\in [v]\in ST^{\ast}\Sigma$ with $h^{\ast}(v,v)=1$ and $\abb{\pi}{ST^{\ast}\Sigma}{\Sigma}$ denotes the projection.
There exist smooth positive functions $\rho_t$ with $\phi_t^{\ast}\alpha=\rho_t\alpha$.
Define $\tilde{C}:=\max\limits_{[0,1]\times ST^{\ast}\Sigma}\rho_t$.
Then for $0<\delta<1$ one has $\phi_t^{\ast}\alpha\leq \tilde{C}\alpha$ for all $t\in [-\delta,\delta]$.
Moreover since $N$ is globally hyperbolic one can choose $\hat{C}$ such that 
$$I^+((-\delta,p))\cap I^-((\delta,p))\subset [-\delta,\delta]\times B_{\hat{C}\delta}(p).$$
Here $B_{\hat{C}\delta}(p)$ denotes the ball of radius $\hat{C}\delta$ around $p$ with respect to $h$.
Let $(f(t),\gamma(t))$ be a causal curve from $(-\delta,p)$ to $(s,q)$.
Then 
\begin{align*}
d_{\alpha}(S((0,p)),S((s,q)))&\leq d_{\alpha}(S((0,p)),S((-\delta,p)))+d_{\alpha}(S((-\delta,p)),S((s,q)))\\
&\leq \integ{0}{1}{\max\limits_{\phi_t(F_p)}|\phi_t^{\ast}\alpha(\hat{\gamma}'(t))|}{t}+C\delta\leq (\tilde{C}\hat{C}+C)\delta.
\end{align*}
\end{proof}

Contrary to the Riemannian case there is no canonical way to associate a metric to a Lorentzian manifold.
So it is not surprising that our construction depends on the choice of a contact form $\alpha$ on $\mathcal{N}_g$.

\begin{qu}
Are there further relations of $d_{\alpha}^X$ and $\tau_{\alpha}|_{S(X)\times S(X)}$ to the causal structure of $N$ and to the Lorentzian metric $g$ for specific choices of the contact form $\alpha$ (see for instance the relation between certain contact forms and gravitational redshift described in \cite{Chernov18})?
\end{qu}

\begin{remark}
An 'intrinsic' way to define a metric and a Lorentzian distance function on $N$ is to only consider Legendrian isotopies that are contained in $S(N)$ in the definition of $d_{\alpha}$ and $\tau_{\alpha}$.
We will denote these functions by $\hat{d}_{\alpha}$ and $\hat{\tau}_{\alpha}$.
Then $d_{\alpha}\leq \hat{d}_{\alpha}$ and $\hat{\tau}_{\alpha}\leq \tau_{\alpha}$.
The proof of Theorem \ref{thmglobhyp} implies that $ \hat{d}_{\alpha}$ induces the manifold topology on $N$.
\end{remark}

\begin{exmp}
Consider the manifold $N=\real{}\times \Sigma$ with Lorentzian metric $g=-dt^2+k$, where $k$ is a complete Riemannian metric on $\Sigma$.
In this case $(N,g)$ is globally hyperbolic (\cite{Minguzzi2}).
Moreover assume that the isotopy class of the fibres in $ST^{\ast}\Sigma$ is orderable.
All null geodesics are up to parametrisation of the form $(t,\beta(t))$, where $\beta(t)$ is a geodesic of $k$ in $\Sigma$.
Thus the path $\phi_t=\rho_0\rho_t^{-1}$ used in the proof of Theorem \ref{thmglobhyp} is given by 
$$\phi_t=\phi_{t}^{\alpha}.$$
Here $-\alpha$ is the contact form induced by $k$ on $ST^{\ast}\Sigma$, i.e. $\phi_t$ is the inverse of the co-geodesic flow of $k$ and is positive with respect to the co-orientation defined by $\alpha$.
Denoting by $F_p$ the fibre over a point $p\in \Sigma$ we have $S(t,p)=\phi_t(F_p)$.
Let $(f(t),\gamma(t))$ be a curve in $N$.
Choose a parametrisation $\abb{l_t}{S^n}{F_{\gamma(t)}}$.
Then for $w\in S((f(t),\gamma(t))$  and $u\in S^{n}$ with $w=\phi_{f(t)}(l_{t}(u))$
\begin{align*}
\alpha_{w}\left(\frac{d}{dt}\phi_{f(t)}(l_t(u))\right)&=(\phi_{f(t)}^{\ast}\alpha)_{l_{t}(u)}(\hat{\gamma}'(t))+f'(t)\alpha_w( X^{\phi}_{f(t)}(w))\\
&=k(l_t(u),\gamma'(t))+f'(t).
\end{align*}
Hence
\begin{align*}
\max\limits_{v\in S(f(t),\gamma(t))}\alpha_{v}(X_t^l(v))=|\gamma'(t)|_k+f'(t)
\end{align*}
and
\begin{align*}
\min\limits_{v\in S(f(t),\gamma(t))}\alpha_{v}(X_t^l(v))=-|\gamma'(t)|_k+f'(t).
\end{align*}
It follows that $\hat{d}_{\alpha}$ is the metric induced by the continuous Finsler-metric $|\cdot|_k+|dt|$ and $\hat{\tau}_{\alpha}$ the Lorentzian distance function of the continuous Lorentz-Finsler metric (\cite{Javaloyes20}) $-|\cdot|_k+dt$.
\end{exmp}

\end{document}